\documentclass{amsart}
\usepackage{amsmath, amssymb, amsthm}
\usepackage{mathtools}
\usepackage{thmtools}
\usepackage{thm-restate}
\usepackage{graphicx}
\usepackage{pinlabel}
\graphicspath{{./Figures/}}
\usepackage{float}
\usepackage[rightcaption]{sidecap}
\usepackage{wrapfig}
\usepackage{caption}
\usepackage{subcaption}
\usepackage[all]{xy}
\usepackage[english]{babel}
\usepackage[dvipsnames]{xcolor}
\usepackage[margin=1in]{geometry}
\usepackage{hyperref}
\hypersetup{colorlinks=true, linkcolor=NavyBlue, citecolor=NavyBlue}
\addto\extrasenglish{}
\numberwithin{equation}{section}
\usepackage[hyperpageref]{backref}
\newtheorem{theorem}{Theorem}[section]
\theoremstyle{definition}
\newtheorem{definition}[theorem]{Definition}

\newtheorem{example}[theorem]{Example}

\newtheorem{remark}[theorem]{Remark}
\newtheorem{question}[theorem]{Question}

\newtheorem{lemma}[theorem]{Lemma}

\begin{document}

\title{An algorithm for Seifert surfaces in $3$-manifolds via surgery presentations}
\author{Geunyoung Kim}
\address{Department of Mathematics \& Statistics, McMaster University, Hamilton, Ontario, Canada}
\email{kimg68@mcmaster.ca}

\begin{abstract}
The classical Seifert algorithm provides an explicit construction of a Seifert surface for any link in $S^3$. Alegria and Menasco extended this construction to integral homology $3$-spheres using Heegaard splittings. In this paper, we extend the Seifert algorithm to null-homologous links in arbitrary $3$-manifolds via surgery on framed links in $S^3$.
\end{abstract}

\maketitle

\section{Introduction}
A \textit{framed link} in $S^3$ is a pair $(L,\phi)$, where $\phi:\bigsqcup_{i=1}^n (S^1\times B^2)_i\hookrightarrow S^3$ is an embedding of a disjoint union of solid tori, and $L=\phi\left(\bigsqcup_{i=1}^n(S^1\times\{(0,0)\})_i\right)$ is a link in $S^3$. The embedding $\phi$ is called a \textit{framing} of $L$. We write $L_i=\phi_i(S^1\times\{(0,0)\})$, where $\phi_i=\phi|_{(S^1\times B^2)_i}$ is the restriction to the $i$-th solid torus, and refer to $(L_i,\phi_i)$ as a \textit{framed knot}.

The result of \textit{surgery} on $S^3$ along the framed link $(L,\phi)$ is the closed, orientable, connected $3$-manifold
\[
S^3(L,\phi)=\left(S^3\setminus\operatorname{int}(\nu(L))\right) \cup_{\phi|_{\bigsqcup_{i=1}^n (S^1\times S^1)_i}}
\left(\bigsqcup_{i=1}^n(B^2\times S^1)_i\right),
\]
where $\nu(L)=\phi(\bigsqcup_{i=1}^n(S^1\times B^2)_i)$ is a tubular neighborhood of $L$ in $S^3$. Equivalently, it is obtained by removing the interior of $\nu(L)$ and gluing in solid tori along the boundary via the restriction of $\phi$. Throughout this paper, we regard $S^3$ as equipped with its standard orientation, and the orientation of $S^3(L,\phi)$ is induced from $S^3$. The manifold $S^3(L,\phi)$ is the boundary of a $4$-manifold obtained by gluing $2$-handles to the $4$-ball $B^4$ along $(L,\phi)$. 

Every closed, orientable, connected $3$-manifold can be obtained by surgery along a framed link in $S^3$ \cite{lickorish1962representation,wallace1960modifications}. This allows us to work entirely in $S^3$ and reduces questions about $3$-manifolds to framed link data.

Let $(L,\phi)$ be a framed link in $S^3$, and let $K\subset S^3\setminus\operatorname{int}(\nu(L))$ be a link. By definition, $S^3\setminus\operatorname{int}(\nu(L))\subset S^3(L,\phi)$, so $K$ may be regarded as a link in the surgered manifold $S^3(L,\phi)$. Conversely, any link in $S^3(L,\phi)$ can be isotoped into the complement $S^3\setminus\operatorname{int}(\nu(L))$ since it can be arranged to be in general position with respect to the dual framed link.

Recall that a \textit{Seifert surface} for an oriented link $K$ in a $3$-manifold $Y$ is a compact, oriented surface $F\subset Y$ without closed components such that $\partial F=K$. The link $K$ admits a Seifert surface if and only if it is null-homologous in $Y$. While this condition characterizes the existence of a Seifert surface, it does not in general provide a concrete method for constructing one.

For a link in $S^3$, Seifert surfaces can be constructed explicitly from a link diagram \cite{seifert1935geschlecht}. The extension to homology $3$-spheres is due to Alegria and Menasco via Heegaard splittings \cite{alegria2024seifert}.

\begin{question}
Is there an explicit algorithm for constructing a Seifert surface for a null-homologous link in a $3$-manifold?
\end{question}

Our main result gives an affirmative answer to this question in the setting of framed link presentations.

\begin{restatable*}{theorem}{Algorithm}\label{thm: Algorithm}
Let $Y=S^3(L,\phi)$ be a $3$-manifold obtained by surgery on a framed link $(L,\phi)$ in $S^3$. Assume that $K\subset S^3\setminus\operatorname{int}(\nu(L))$ is an oriented null-homologous link in $Y$. Then there exists an explicit algorithm which isotopes $K$ in $Y$ to a link $K'$ that bounds a Seifert surface in $S^3\setminus\operatorname{int}(\nu(L))$.
\end{restatable*}

We recall that the \textit{linking number} $\operatorname{lk}_Y(K_1,K_2)$ of oriented null-homologous knots $K_1,K_2$ in an oriented $3$-manifold $Y$ is defined as the algebraic intersection number of a Seifert surface for $K_1$ with $K_2$. Since both knots are null-homologous, this number is independent of the choice of Seifert surface and is symmetric in $K_1$ and $K_2$. When $Y=S^3$, the linking number can be easily read off from a diagram without constructing a Seifert surface; it is equal to one half the sum of the signs of all crossings between the two components.

This allows us to compute the linking number diagrammatically in a $3$-manifold presented by surgery on a framed link in $S^3$.

\begin{restatable*}{corollary}{linkingformula}\label{cor: linking formula}
Let $Y=S^3(L,\phi)$ be a $3$-manifold obtained by surgery on a framed link $(L,\phi)$ in $S^3$. Assume that $K_1,K_2\subset S^3\setminus\operatorname{int}(\nu(L))$ are oriented null-homologous knots in $Y$. Then the linking number between $K_1$ and $K_2$ in $Y$ satisfies
\[
\operatorname{lk}_Y(K_1,K_2)=\operatorname{lk}_{S^3}(K_1,K_2)-X_{K_1}^{\intercal}V_{K_2},
\]
where $M_{(L,\phi)}$ denotes the linking matrix of $(L,\phi)$, $V_{K_i}$ is the linking vector of $K_i$ with $L$, and $X_{K_i}$ is any solution of the linear system $M_{(L,\phi)} X_{K_i}=V_{K_i}$. The right-hand side is independent of the choice of the solution $X_{K_1}$. In particular, if $Y$ is a homology $3$-sphere, then 
\[
\operatorname{lk}_Y(K_1,K_2)=\operatorname{lk}_{S^3}(K_1,K_2)-V_{K_1}^{\intercal}M_{(L,\phi)}^{-1} V_{K_2}.
\]
\end{restatable*}

As an application of \autoref{thm: Algorithm} and \autoref{cor: linking formula}, we can construct Seifert surfaces and compute the associated Seifert matrices, signatures, and Alexander polynomials for knots in homology $3$-spheres; see \autoref{ex: example 1} and \autoref{ex: example 2}. 

Let $(K,\phi)$ be a null-homologous framed knot in an oriented $3$-manifold $Y$, where $\phi:S^1\times B^2\hookrightarrow Y$ satisfies $\phi(S^1\times\{(0,0)\})=K$. The \textit{framing coefficient} (or \textit{surgery coefficient}) of $K$ in $Y$ is defined by $\operatorname{fr}_Y(K)=\operatorname{lk}_Y(K,\widetilde K)$, where $\widetilde K=\phi(S^1\times\{\operatorname{pt}\})$ for some $\mathrm{pt}\in S^1$ is a \textit{push-off} of $K$ determined by the framing $\phi$, and $K$ and $\widetilde K$ are oriented consistently. We denote the result of surgery $Y(K,\phi)$ by $Y(K^p)$, where $p=\operatorname{fr}_Y(K)$, and refer to this operation as \textit{$p$-surgery} on $K$ in $Y$.

Framings of a knot are classified by $\pi_1(SO(2))\cong\mathbb{Z}$. When $K$ is null-homologous in $Y$, a choice of Seifert surface determines a preferred reference framing, and the integer $p=\operatorname{lk}_Y(K,\widetilde K)$ measures the framing relative to this reference.

When $Y=S^3$, a framed link $(L,\phi)$ is represented by a link diagram $L=L_1^{p_1}\cup\cdots\cup L_n^{p_n}\subset S^3$ together with framing coefficients $\operatorname{fr}_{S^3}(L_i)=p_i$ assigned to each component $L_i$. Such a diagram is called a \textit{surgery diagram}, and the resulting surgered manifold is denoted by $S^3(L_1^{p_1}\cup\cdots\cup L_n^{p_n})$. Two surgery diagrams represent diffeomorphic $3$-manifolds if and only if they are related by a sequence of handle slides and blow-ups or blow-downs \cite{kirby1978calculus}.

We ask how surgery on a knot in a surgered manifold can be described in terms of the original surgery diagram.

\begin{question}
Let $K\subset S^3\setminus\operatorname{int}(\nu(L))$ be a null-homologous knot in $Y=S^3(L,\phi)$. Describe a surgery diagram for the $p$-surgery $Y(K^p)$.
\end{question}

The following corollary provides an explicit answer.

\begin{restatable*}{corollary}{framingcoefficient}\label{cor: Kirby}
\label{cor:surgery-coefficient-conversion}
Let $K\subset S^3\setminus\operatorname{int}(\nu(L))$ be a null-homologous knot in $Y=S^3(L_1^{p_1}\cup\dots\cup L_n^{p_n})$. Then a surgery diagram for the $p$-surgery $Y(K^p)$ is obtained by assigning to $K$ a framing coefficient $p+X_K^{\intercal}V_K$ to $K$ in the surgery diagram of $Y$. In other words, 
\[Y(K^p)\cong S^3(L_1^{p_1}\cup \cdots\cup L_n^{p_n}\cup K^{p+X_K^{\intercal}V_K}).\]
\end{restatable*}

In \autoref{ex: example 1} and \autoref{ex: example 2}, we apply \autoref{cor: Kirby} to show that certain knots in homology $3$-spheres are not local knots, that is, they are not isotopic to knots contained in a $3$-ball. More precisely, if $K\subset Y$ were isotopic to a knot $J$ in a $3$-ball in $Y$, then for every integer $p$ we would have $Y(K^p)\cong Y\#S^3(J^p)$. This yields a contradiction: for some $p$, elementary Kirby calculus shows that the two manifolds are not diffeomorphic.

\subsection*{Acknowledgments} The author would like to thank Hans Boden, Ducan McCoy, and William Menasco for helpful conversations.

\section{Preliminaries}\label{sec: some algebra}
In this section we introduce the algebraic tools needed for the Seifert surface algorithm, including linking matrices, homology computations, and slides of links over framed links.

Let $(L,\phi)$ be an oriented framed link in $S^3$, where $L=L_1\cup\cdots\cup L_n$. For each $i$, let $\phi_i:S^1\times B^2\hookrightarrow S^3$ denote the framing of $L_i$. A \textit{meridian} of $L_i$ is the curve $\mu_i=\phi_i(\{x\}\times \partial B^2)\subset \partial \nu(L_i)$, for some $x\in S^1$, oriented so that $\operatorname{lk}(L_i,\mu_i)=1$. Here, $\operatorname{lk}$ denotes a linking number in $S^3$. A \textit{push-off} of $L_i$ is the curve $\widetilde L_i=\phi_i(S^1\times\{y\})$, for some $y\in\partial B^2$, oriented so that it agrees with the orientation of $L_i$. We write $-\widetilde L_i$ for the same push-off with the opposite orientation.

\begin{definition}
The \textit{linking matrix} of the oriented framed link $(L,\phi)$ is the $n\times n$ integer matrix
\[
M_{(L,\phi)}=(m_{ij}),\qquad m_{ij}=\operatorname{lk}(L_i,\widetilde L_j).
\]
\end{definition}
In particular, the matrix $M_{(L,\phi)}$ is symmetric. For $i\neq j$, we have $\operatorname{lk}(L_i,\widetilde L_j)=\operatorname{lk}(L_i,L_j)$. The diagonal entry $m_{ii}$ agrees with the framing coefficient of $L_i$.

The meridians $\mu_1,\dots,\mu_n$ generate the first homology group
\[
H_1\!\left(S^3 \setminus \operatorname{int}(\nu(L))\right) \cong \mathbb{Z}^n.
\]
When surgery is performed along $(L,\phi)$, the attaching curve for the solid torus corresponding to $L_i$ represents the homology class
\[
[\widetilde L_i]=p_i[\mu_i]+\sum_{j\neq i}\operatorname{lk}(L_i,L_j)[\mu_j].
\]
Since each attaching curve bounds a meridional disk in the surgered manifold, these classes become null-homologous in $S^3(L,\phi)$. Consequently, the first homology group of the surgered manifold is given by
\[
H_1(S^3(L,\phi))\cong\mathbb{Z}^n/M_{(L,\phi)}\mathbb{Z}^n .
\]
In particular, 
\[
S^3(L,\phi)\ \text{is an integral homology $3$-sphere} \quad\Longleftrightarrow\quad \det(M_{(L,\phi)})=\pm 1.
\]

Let $K=K_1\cup\cdots\cup K_m\subset S^3\setminus\operatorname{int}(\nu(L))\subset S^3(L,\phi)$ be an oriented link, where the orientation of $S^3(L,\phi)$ is induced by the standard orientation of $S^3$.

\begin{definition}
The \textit{linking vector} of $K$ with respect to $(L,\phi)$ is the column vector
\[
V_K=(v_1,\dots,v_n)^\intercal\in\mathbb{Z}^n, \qquad v_i=\operatorname{lk}(K,L_i)=\sum_{t=1}^m\operatorname{lk}(K_t,L_i).
\]
\end{definition}

Under the identification $H_1\!\left(S^3\setminus\operatorname{int}(\nu(L))\right)\cong \mathbb{Z}^n$ sending the meridian class $[\mu_i]$ to the $i$-th standard basis vector, the homology class $[K] = [K_1]+\cdots+[K_m]$ corresponds to the vector $V_K$. We then obtain the following.
\[
[K]=0\ \text{in}\ H_1(S^3(L,\phi))\quad\Longleftrightarrow\quad
M_{(L,\phi)} X_K = V_K\qquad \text{for some } X_K \in\mathbb{Z}^n.
\]
If $S^3(L,\phi)$ is a homology $3$-sphere, that is $\operatorname{det}(M_{(L,\phi)})=\pm1$, then the solution $X_K=M_{(L,\phi)}^{-1} V_K$ is unique.

\begin{definition}
A vector $X_K\in\mathbb Z^n$ satisfying $M_{(L,\phi)}X_K=V_K$ is called a \textit{solution vector} of $K$.
\end{definition}

A solution vector $X_K$ expresses the homology class of $K$ as a linear combination of the attaching curves $\widetilde L_1,\dots,\widetilde L_n$.

We now define certain isotopies of $K$ in $S^3(L,\phi)$.

\begin{definition}
Let $K_t$ be a component of $K$ and $L_i$ a component of $L$. A \textit{positive} (resp. \textit{negative}) \textit{slide} of $K_t$ over $L_i$ is the band connected sum
\[
K_t\#_b\widetilde{L}_i ,\quad(\text{resp.}\quad K_t\#_b(-\widetilde{L}_i))
\]
where the band $b$ is chosen so that the interior of the band is disjoint from $\nu(L)$ and the orientation on the band sum agrees with the given orientations of $K_t$ and $\widetilde{L}_i$ (resp. $-\widetilde{L}_i$ ). A \textit{positive} (resp.\ \textit{negative}) \textit{slide} of the link $K$ over $L_i$ means performing a positive (resp.\ negative) slide of a component of $K$ over $L_i$.
\end{definition}

\begin{lemma}\label{lem: slide of link}
Let $K\subset S^3\setminus \operatorname{int}(\nu(L))$ be an oriented link, and let $X_K\in \mathbb{Z}^n$ be a solution vector satisfying $M_{(L,\phi)}X_K=V_K$. Then there exists a link $K'\subset S^3\setminus \operatorname{int}(\nu(L))$ obtained from $K$ by a sequence of slides over components of $L$ such that $K$ is isotopic to $K'$ in the surgered manifold $S^3(L,\phi)$ and has linking vector $V_{K'}=\mathbf{0}$. Equivalently, $\operatorname{lk}(K',L_i)=0$ for all $i$.
\end{lemma}

\begin{proof}
Write $X_K=(x_1,\dots,x_n)^\intercal\in\mathbb{Z}^n$. For each $i$, we modify $K$ by sliding it over $L_i$ as follows: if $x_i>0$, perform $x_i$ negative slides of $K$ over $L_i$; if $x_i<0$, perform $|x_i|$ positive slides of $K$ over $L_i$.

Let $K'$ denote the link obtained after performing all such slides. By construction, $K'$ is isotopic to $K$ in the surgered manifold $S^3(L,\phi)$ and may be regarded as lying in $S^3\setminus \operatorname{int}(\nu(L))$. With our convention, a positive slide of $K$ over $L_i$ adds the $i$-th column of $M_{(L,\phi)}$ to the linking vector, while a negative slide subtracts it. Since $M_{(L,\phi)}X_K=V_K$, these slides change the linking vector by $-M_{(L,\phi)}X_K$, and hence $V_{K'}=V_K-M_{(L,\phi)}X_K=\mathbf{0}$.
\end{proof}

\begin{remark}
At each step of the construction, the choice of bands is not unique, and hence the resulting links need not be isotopic to $K$ in $S^3$. Nevertheless, they are all isotopic to $K$ in the surgered manifold $S^3(L,\phi)$.
\end{remark}

\begin{lemma}\label{lem: tubing}
Let $K'\subset S^3\setminus \operatorname{int}\nu(L)$ be an oriented link satisfying $\operatorname{lk}(K',L_i)=0$ for all $i$. If $F\subset S^3$ is a Seifert surface for $K'$, then there exists a Seifert surface $F'$ for $K'$ such that $F'\subset S^3\setminus\operatorname{int}(\nu(L))$.
\end{lemma}

\begin{proof}
Let $F\subset S^3$ be a Seifert surface for $K'$. After isotoping $F$ rel $\partial F=K'$, we may assume that $F$ is transverse to each component $L_i$. Since $\operatorname{lk}(K',L_i)=0$ for all $i$, the algebraic intersection number satisfies $F\cdot L_i = 0$. Hence the points of $F\cap L_i$ can be paired into oppositely signed pairs.

Fix $i$ and let $p_-,p_+\in F\cap L_i$ be a pair of intersection points with opposite signs. Remove small embedded disks from $F$ around $p_-$ and $p_+$. Let $\alpha\subset L_i$ be the subarc connecting $p_-$ and $p_+$. Attach a tube $S^1\times I$ to the punctured $F$ along $\alpha$, choosing the annulus so that its interior is disjoint from $\nu(L)$. This tubing eliminates the pair $p_-,p_+$ without changing $\partial F=K'$.

Performing this construction for all such pairs and all components $L_i$ yields a Seifert surface $F'$ for $K'$ satisfying $F'\subset S^3\setminus\operatorname{int}\nu(L)$.
\end{proof}

\begin{remark}
One may obtain a Seifert surface $F$ for $K'$ in $S^3$ using the classical Seifert algorithm \cite{seifert1935geschlecht}, or start with any chosen Seifert surface if one is already available. Moreover, each tube corresponds decreases the Euler characteristic by $2$. Thus $\chi(F')=\chi(F)-2(\#\text{ of tubes})$. In particular, if $F$ is connected, then $g(F')=g(F)+(\#\text{ of tubes})$.
\end{remark}

\section{Main theorem}\label{sec: main}
In this section, we present a Seifert algorithm for null-homologous links in $3$-manifolds, a linking number formula, and a method for constructing a surgery diagram for surgery on a null-homologous knot in a $3$-manifold. In the examples, we construct Seifert surfaces for knots in homology $3$-spheres and compute their Seifert matrices, signatures, and Alexander polynomials using the algorithm and the linking number formula. We construct surgery diagrams for surgeries on these knots, and show that they are not isotopic into a $3$-ball by applying elementary Kirby calculus on the surgery diagrams.

\Algorithm

\begin{proof}
Since $K$ is null-homologous in $Y=S^3(L,\phi)$, there exists a solution vector $X_K\in\mathbb Z^n$ satisfying $M_{(L,\phi)}X_K=V_K$. By \autoref{lem: slide of link}, using the vector $X_K$, the link $K$ can be explicitly isotoped in $Y$ to a link $ K'\subset S^3\setminus\operatorname{int}(\nu(L))$ with $\operatorname{lk}(K',L_i)=0$ for all $i$. Applying \autoref{lem: tubing} to $K'$, any Seifert surface $F$ for $K'$ in $S^3$ can be modified by tubing to obtain a Seifert surface
$F'\subset S^3\setminus \operatorname{int}\nu(L)$. Here, $F$ may be constructed using the classical Seifert algorithm in $S^3$.
\end{proof}

\linkingformula

\begin{proof}
Apply \autoref{thm: Algorithm} to $K_1$. This yields a knot $K_1'\subset S^3\setminus\operatorname{int}(\nu(L))$, obtained from $K_1$ by slides over $L$, and a Seifert surface $F_1'\subset S^3\setminus\operatorname{int}(\nu(L))\subset Y$ with $\partial F_1'=K_1'$. Since slides are realized by isotopies in $Y$,
\[
\operatorname{lk}_Y(K_1,K_2)=\operatorname{lk}_Y(K_1',K_2)=F_1'\cdot K_2=\operatorname{lk}_{S^3}(K_1',K_2),
\]
where the last equality holds since $F_1'\subset S^3\setminus\operatorname{int}(\nu(L))$.

Let $V_{K_i}\in\mathbb Z^n$ be the linking vectors and write $X_{K_1}=(x_1,\dots,x_n)^\intercal\in\mathbb{Z}^n$ for a solution vector. To obtain the knot $K_1'$, if $x_j<0$ (resp.\ $x_j>0$), we perform $|x_j|$ positive (resp.\ $x_j$ negative) slides over $L_j$. A positive (resp.\ negative) slide of $K_1$ over $L_j$ changes $\operatorname{lk}_{S^3}(K_1,K_2)$ by $+\operatorname{lk}_{S^3}(K_2,L_j)$ (resp.\ $-\operatorname{lk}_{S^3}(K_2,L_j)$); with our convention for the solution vector $X_{K_1}$, the total change is $-X_{K_1}^{\intercal}V_{K_2}$.

Hence,
\[
\operatorname{lk}_Y(K_1,K_2)=\operatorname{lk}_{S^3}(K_1',K_2)=\operatorname{lk}_{S^3}(K_1,K_2)-X_{K_1}^{\intercal}V_{K_2}.
\]
It remains to show that the right-hand side is independent of the choice of the solution $X_{K_1}$. If $X$ and $X'$ are two solutions, then $M_{(L,\phi)}(X-X')=\mathbf{0}$. Since $V_{K_2}\in\operatorname{im} (M_{(L,\phi)})$ and $M_{(L,\phi)}$ is symmetric, we have $(X-X')^{\intercal}V_{K_2}=0$.

If $Y$ is a homology $3$-sphere, then $M_L$ is unimodular and $X_{K_1}=M_{(L,\phi)}^{-1}V_{K_1}$, hence
\[
\operatorname{lk}_Y(K_1,K_2)=\operatorname{lk}_{S^3}(K_1,K_2)-
V_{K_1}^{\intercal} M_{(L,\phi)}^{-1} V_{K_2}.
\]
\end{proof}

\framingcoefficient

\begin{proof}
Let $\widetilde K$ be a push-off of $K$. By definition, $p=\operatorname{fr}_Y(K)=\operatorname{lk}_Y(K,\widetilde K)$. Applying \autoref{cor: linking formula} to the pair $(K,\widetilde K)$ gives $\operatorname{lk}_Y(K,\widetilde K) =\operatorname{lk}_{S^3}(K,\widetilde K) -X_K^{\intercal}V_{\widetilde K}$. Since $\widetilde K$ is a parallel push-off of $K$, we have $V_{\widetilde K}=V_K$. Thus the framing coefficient of $K$ in $S^3$ is $\operatorname{lk}_{S^3}(K,\widetilde K)=p+X_K^\intercal V_K$.
\end{proof}

\begin{example}\label{ex: example 1}
Let $K$ be an oriented knot in $Y=S^3(L_1^1\cup L_2^1\cup L_3^1)$ shown in \autoref{fig: example} (a). The linking matrix of $L=L_1^1\cup L_2^1\cup L_3^1$ is
\[
M_L=
\begin{pmatrix}
1&0&0\\
0&1&0\\
0&0&1
\end{pmatrix},
\]
and hence $Y$ is a homology $3$-sphere. The linking vector of $K$ with respect to $L$ is
\[
V_K=
\begin{pmatrix}
-1\\
0\\
0
\end{pmatrix}.
\]
It follows that the solution vector is
\[
X_K=M_L^{-1}V_K=
\begin{pmatrix}
-1\\
0\\
0
\end{pmatrix}.
\]
We apply the algorithm in \autoref{thm: Algorithm}. Performing a positive slide of $K$ over $L_1$ yields a knot $K'$ in \autoref{fig: example} (b) satisfying $V_{K'}=\mathbf{0}$. Let $F$ be a Seifert surface for $K'$ in $S^3$; in this case, $F$ is a disk. By tubing $F$ along a subarc of $L_2$ connecting two intersection points of opposite sign, we obtain a Seifert surface $F'\subset S^3\setminus\operatorname{int}(\nu(L))\subset Y$.

Let $\alpha,\beta\subset F'$ be generators of $H_1(F')$. Since $M_L$ is the identity matrix, the linking vectors and solution vectors coincide. The curves $\alpha$ and $\beta$ are meridians of $L_2$ and $L_3$, respectively.
In particular,
\[
V_\alpha=V_{\alpha^+}=X_\alpha=X_{\alpha^+}=
\begin{pmatrix}
0\\
1\\
0
\end{pmatrix},
\qquad
V_\beta=V_{\beta^+}=X_\beta=X_{\beta^+}=
\begin{pmatrix}
0\\
0\\
1
\end{pmatrix},
\]
where $\alpha^+$ (resp.\ $\beta^+$) denotes a push-off of $\alpha$
(resp.\ $\beta$) in the positive normal direction.
Moreover,
\[
\operatorname{lk}_{S^3}(\alpha,\alpha^+)=0,\quad
\operatorname{lk}_{S^3}(\alpha,\beta^+)=0,\quad
\operatorname{lk}_{S^3}(\beta,\alpha^+)=1,\quad
\operatorname{lk}_{S^3}(\beta,\beta^+)=0.
\]

Applying the linking number formula \autoref{cor: linking formula}, the Seifert matrix with respect to the ordered basis $(\alpha,\beta)$ is
\[
A=
\begin{pmatrix}
\operatorname{lk}_Y(\alpha,\alpha^+) &
\operatorname{lk}_Y(\alpha,\beta^+)\\
\operatorname{lk}_Y(\beta,\alpha^+) &
\operatorname{lk}_Y(\beta,\beta^+)
\end{pmatrix}
=
\begin{pmatrix}
\operatorname{lk}_{S^3}(\alpha,\alpha^+)-X_\alpha^\intercal V_{\alpha^+} &
\operatorname{lk}_{S^3}(\alpha,\beta^+)-X_\alpha^\intercal V_{\beta^+}\\
\operatorname{lk}_{S^3}(\beta,\alpha^+)-X_\beta^\intercal V_{\alpha^+} &
\operatorname{lk}_{S^3}(\beta,\beta^+)-X_\beta^\intercal V_{\beta^+}
\end{pmatrix}
=
\begin{pmatrix}
-1 & 0\\
1 & -1
\end{pmatrix}.
\]
Consequently, the signature is
\[
\sigma_Y(K)=\operatorname{sign}(A+A^T)
=\operatorname{sign}
\begin{pmatrix}
-2&1\\
1&-2
\end{pmatrix}
=-2
\]
and the Alexander polynomial (up to multiplication by $\pm t^k$) is
\[
\Delta_Y(K)\doteq\det(A-tA^\intercal)
=
\det
\begin{pmatrix}
-1+t & -t\\
1 & -1+t
\end{pmatrix}
=t^2-t+1.
\]
Thus $K$ does not bound a disk in $Y$ and the minimal genus of $K$ in $Y$ satisfies $g_Y(K)=1$.

We next show that $K$ is not a \textit{local knot} in $Y$, that is, $K$ is not isotopic to any knot contained in a $3$-ball in $Y$. Suppose for contradiction that $K$ is isotopic to a knot $J \subset B^3 \subset Y$. Then for any integer $p$,
\[
Y(K^p) \cong Y \# S^3(J^p).
\]

By \autoref{cor: Kirby}, a surgery diagram for $Y_{-1}(K)$ is obtained from the diagram in \autoref{fig: example} (a) by assigning to $K$ the integer 
\[
0=(-1)+1=p+X_K^\intercal V_K.
\]
A straightforward Kirby calculus shows that this diagram represents $S^3$, so
\[
S^3\cong Y(K^{-1})\cong Y\#S^3(J^{-1}).
\]
It follows that $Y\cong S^3$, which is impossible because $\pi_1(Y)$ is nontrivial. Hence $K$ is not a local knot.

We give an alternative argument that $Y \not\cong S^3$. Suppose that $Y \cong S^3$. Since $K$ is nontrivial in $Y$, the Property P theorem \cite{kronheimer2004witten} implies that $Y(K^{-1}) \ncong S^3$. This contradicts the fact that $Y(K^{-1}) \cong S^3$.
\end{example}

\begin{example}\label{ex: example 2}
Let $M$ be a contractible $4$-manifold represented by the Kirby diagram in \autoref{fig: example} (c). The boundary $Y=\partial M$ admits a surgery diagram obtained by replacing the dotted circle $L_1$ with a $0$-framed component. Let $K$ be an oriented knot in $Y = S^3(L_1^0 \cup L_2^0)$.

The linking matrix of $L=L_1^0\cup L_2^0$ is
\[
M_L=
\begin{pmatrix}
0 & 1\\
1 & 0
\end{pmatrix},
\]
and hence $Y$ is an integral homology $3$-sphere.
The linking vector of $K$ with respect to $L$ is
\[
V_K=
\begin{pmatrix}
0\\
-1
\end{pmatrix}.
\]
The solution vector is
\[
X_K=
M_L^{-1}V_K
=
\begin{pmatrix}
-1\\
0
\end{pmatrix}.
\]

Then performing a positive slide over $L_1$ yields a knot $K'$ in \autoref{fig: example} (d). Let $F'\subset Y$ be the genus one Seifert surface constructed using \autoref{thm: Algorithm}. For the generating curves $\alpha,\beta\subset F'$ of $H_1(F')$, we have
\[
V_\alpha=V_{\alpha^+}=V_\beta=V_{\beta^+}=
\begin{pmatrix}
0\\
1
\end{pmatrix},
\qquad
X_\alpha=X_{\alpha^+}=X_\beta=X_{\beta^+}=
\begin{pmatrix}
1\\
0
\end{pmatrix}.
\]
Moreover,
\[
\operatorname{lk}_{S^3}(\alpha,\alpha^+)=0,\quad
\operatorname{lk}_{S^3}(\alpha,\beta^+)=0,\quad
\operatorname{lk}_{S^3}(\beta,\alpha^+)=1,\quad
\operatorname{lk}_{S^3}(\beta,\beta^+)=0.
\]

Applying the linking number formula \autoref{cor: linking formula} in $Y$, the Seifert matrix with respect to the ordered basis $(\alpha,\beta)$ is
\[
A=
\begin{pmatrix}
\operatorname{lk}_Y(\alpha,\alpha^+) &
\operatorname{lk}_Y(\alpha,\beta^+)\\
\operatorname{lk}_Y(\beta,\alpha^+) &
\operatorname{lk}_Y(\beta,\beta^+)
\end{pmatrix}
=
\begin{pmatrix}
\operatorname{lk}_{S^3}(\alpha,\alpha^+)-X_\alpha^\intercal V_{\alpha^+} &
\operatorname{lk}_{S^3}(\alpha,\beta^+)-X_\alpha^\intercal V_{\beta^+}\\
\operatorname{lk}_{S^3}(\beta,\alpha^+)-X_\beta^\intercal V_{\alpha^+} &
\operatorname{lk}_{S^3}(\beta,\beta^+)-X_\beta^\intercal V_{\beta^+}
\end{pmatrix}
=
\begin{pmatrix}
-1 & 0\\
1 & -1
\end{pmatrix}.
\]
Consequently, the signature is
\[
\sigma_Y(K)=\operatorname{sign}(A+A^T)=-2,
\]
and the Alexander polynomial (up to multiplication by $\pm t^k$) is
\[
\Delta_Y(K)\doteq \det(A-tA^T)=t^2-t+1.
\]
Thus the minimal genera of $K$ in $Y$ and $M$ satisfy $g_Y(K)=1$ and $g_M(Y)=0$, respectively. Indeed, $K$ bounds an obvious embedded disk in $M$, obtained by pushing a trivial disk bounded by $K$ in $S^3$ into the interior of the $0$-handle $B^4\subset M$.

Next, we consider $0$-surgery on $K$ in $Y$. By \autoref{cor: Kirby}, $Y(K^0)\cong S^3(L_1^0\cup L_2^0\cup K^0)$. Here the framing coefficient of $K$ in $S^3$ is
\[
0=0+0=p+X_K^\intercal V_K.
\] A simple Kirby calculus computation shows that $Y(K^0)\cong S^1\times S^2$.

Suppose that $K$ were isotopic to a knot $J$ contained in a $3$-ball in $Y$. Then
\[
Y(K^0)\cong Y\#S^3(J^0).
\]
Since $Y(K^0)\cong S^1\times S^2$, this would imply that $Y \cong S^3$. However, $Y$ is a non-simply connected homology sphere, a contradiction. Thus, $K$ is not a local knot.

Alternatively, suppose that $Y \cong S^3$. Since $K$ is nontrivial in $Y$, Gabai's Property R theorem \cite{gabai1987foliations} implies that $Y(K^0) \not\cong S^1 \times S^2$. This contradicts the fact that $Y(K^0) \cong S^1 \times S^2$.
\end{example}

\begin{figure}[ht!]
\centering

\makebox[\textwidth][c]{
  \begin{minipage}[t]{0.28\textwidth}
  \centering
  \labellist
  \small\hair 2pt
  \pinlabel {$L_3$}  at 10 333
  \pinlabel {$L_2$}  at 30 200
  \pinlabel {$L_1$}  at 10 30
  \pinlabel {$1$}  at 250 30
  \pinlabel {$1$}  at 225 220
  \pinlabel {$1$}  at 250 330
  \pinlabel {$K$} at 120 33
  \pinlabel {\rotatebox{10}{\scalebox{0.8}{$\boldsymbol{>}$}}} at 200.5 22.5
  \pinlabel {\rotatebox{110}{\scalebox{0.8}{$\boldsymbol{>}$}}} at 124 10
  \pinlabel {\rotatebox{90}{\scalebox{0.8}{$\boldsymbol{>}$}}} at 218.7 200
  \pinlabel {\rotatebox{200}{\scalebox{0.8}{$\boldsymbol{>}$}}} at 231 271
  \endlabellist
  \includegraphics[width=\linewidth]{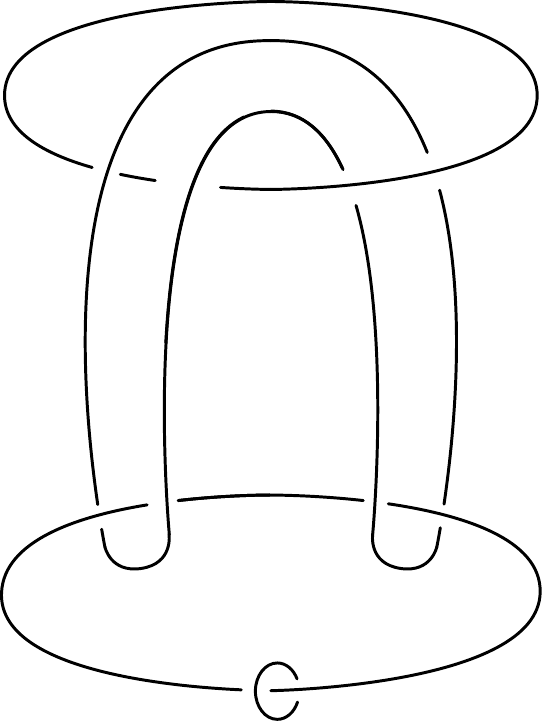}
  \par\smallskip{\small (a)}
  \end{minipage}
  \hspace{0.15\textwidth}
  
  \begin{minipage}[t]{0.28\textwidth}
  \centering
  \labellist
  \small\hair 2pt
  \pinlabel {$L_3$}  at 10 319
  \pinlabel {$L_2$}  at 30 186
  \pinlabel {$L_1$}  at 10 14
  \pinlabel {$1$}  at 250 14
  \pinlabel {$1$}  at 225 204
  \pinlabel {$1$}  at 250 314
  \pinlabel {$K'$} at 120 135
  \pinlabel {\textcolor{red}{$\alpha$}} at 125 297
  \pinlabel {\textcolor{blue}{$\beta$}} at 140 168
  \pinlabel {\textcolor{red}{\rotatebox{270}{\scalebox{0.8}{$\boldsymbol{>}$}}}} at 122 280
  \pinlabel {\textcolor{blue}{\rotatebox{180}{\scalebox{0.8}{$\boldsymbol{>}$}}}} at 120 161
  \pinlabel {\rotatebox{1}{\scalebox{0.8}{$\boldsymbol{>}$}}} at 150 145
  \pinlabel {\rotatebox{10}{\scalebox{0.8}{$\boldsymbol{>}$}}} at 200.5 8.5
  \pinlabel {\rotatebox{90}{\scalebox{0.8}{$\boldsymbol{>}$}}} at 218.7 186
  \pinlabel {\rotatebox{200}{\scalebox{0.8}{$\boldsymbol{>}$}}} at 231 257
  \endlabellist
  \raisebox{3.3mm}{
  \includegraphics[width=\linewidth]{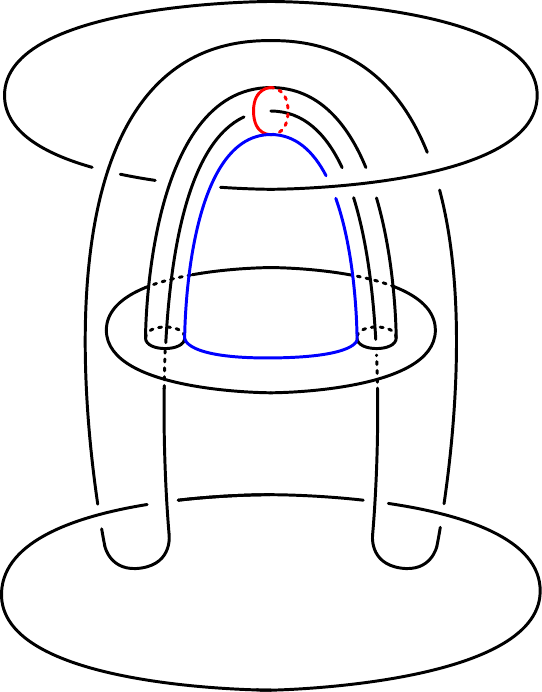}}
  \par\smallskip{\small (b)}
  \end{minipage}}
\bigskip
\medskip

\makebox[\textwidth][c]{
  \begin{minipage}[t]{0.28\textwidth}
  \centering
  \labellist
  \small\hair 2pt
  \pinlabel {$L_2$}  at 30 350
  \pinlabel {$L_1$}  at 5 110
  \pinlabel {$0$}  at 270 370
  \pinlabel {$K$} at 50 130
  \pinlabel {\scalebox{1.5}{$\bullet$}} at 100 96
  \pinlabel {\rotatebox{130}{\scalebox{0.8}{$\boldsymbol{>}$}}} at 21.5 140
  \pinlabel {\rotatebox{1}{\scalebox{0.8}{$\boldsymbol{>}$}}} at 130 96.5
  \pinlabel {\rotatebox{90}{\scalebox{0.8}{$\boldsymbol{>}$}}} at 239.5 160
  \endlabellist
  \includegraphics[width=\linewidth]{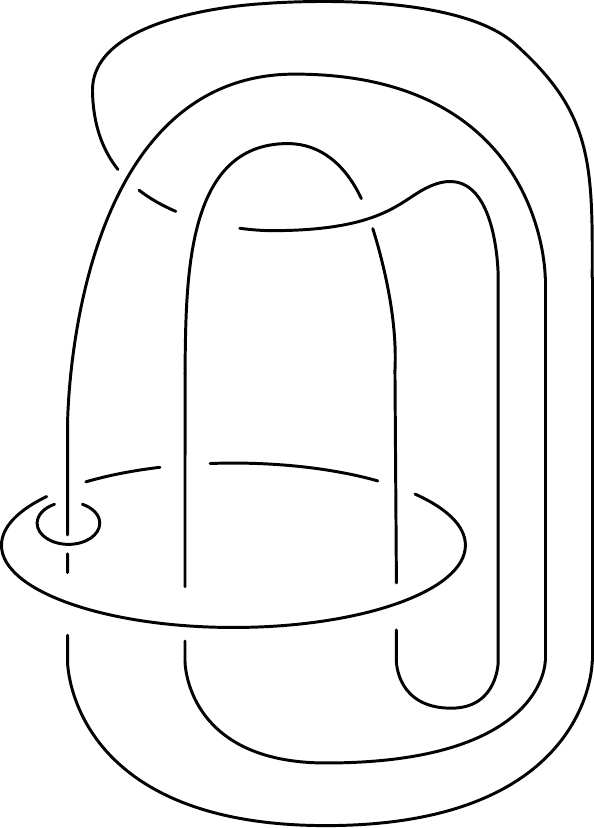}
  \par\smallskip{\small (c)}
  \end{minipage}
  \hspace{0.15\textwidth}
  
  \begin{minipage}[t]{0.28\textwidth}
  \centering
  \labellist
  \small\hair 2pt
  \pinlabel {$L_2$}  at 30 350
  \pinlabel {$L_1$}  at 5 110
  \pinlabel {$0$}  at 160 90
  \pinlabel {$0$}  at 270 370
  \pinlabel {$K'$} at 210 195
  \pinlabel {\textcolor{red}{$\alpha$}} at 133 348
  \pinlabel {\textcolor{blue}{$\beta$}} at 140 219
  \pinlabel {\textcolor{red}{\scalebox{0.8}{\rotatebox{270}{$\boldsymbol{>}$}}}} at 130 330
  \pinlabel {\textcolor{blue}{\rotatebox{180}{\scalebox{0.8}{$\boldsymbol{>}$}}}} at 120 211.5
  \pinlabel {\rotatebox{1}{\scalebox{0.8}{$\boldsymbol{>}$}}} at 150 194
  \pinlabel {\rotatebox{1}{\scalebox{0.8}{$\boldsymbol{>}$}}} at 130 96.5
  \pinlabel {\rotatebox{90}{\scalebox{0.8}{$\boldsymbol{>}$}}} at 239.5 160
  \endlabellist
  \includegraphics[width=\linewidth]{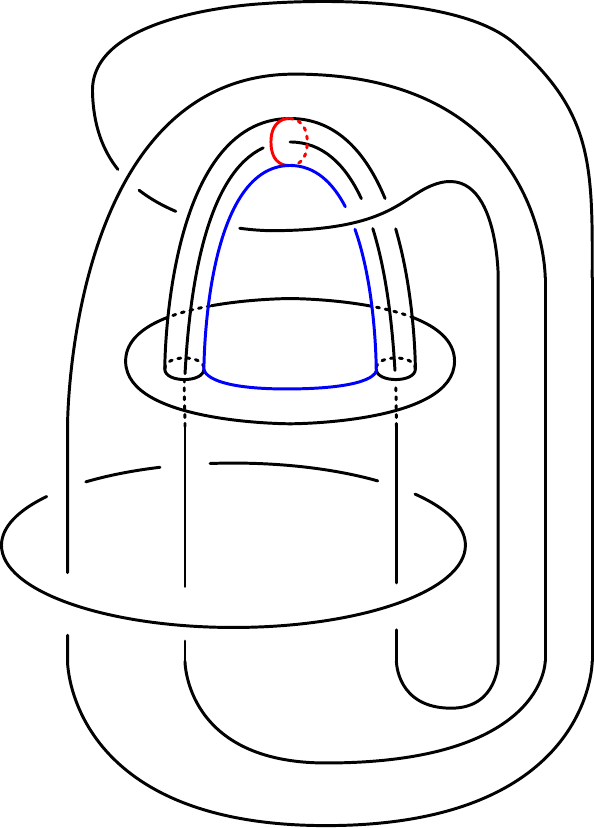}
  \par\smallskip{\small (d)}
  \end{minipage}}

\caption{}
\label{fig: example}
\end{figure}

\bibliographystyle{alpha}
\bibliography{refs}

@article{alegria2024seifert,
  title={A Seifert algorithm for integral homology spheres},
  author={Alegria, Linda V and Menasco, William W},
  journal={arXiv preprint arXiv:2405.14805},
  year={2024}
}

@article{seifert1935geschlecht,
  title={{\"U}ber das geschlecht von knoten},
  author={Seifert, Herbert},
  journal={Mathematische Annalen},
  volume={110},
  number={1},
  pages={571--592},
  year={1935},
  publisher={Springer}
}

@article{lickorish1962representation,
  title={A representation of orientable combinatorial 3-manifolds},
  author={Lickorish, WB Raymond},
  journal={Annals of Mathematics},
  pages={531--540},
  year={1962},
  publisher={JSTOR}
}

@article{wallace1960modifications,
  title={Modifications and cobounding manifolds},
  author={Wallace, Andrew H},
  journal={Canadian Journal of Mathematics},
  volume={12},
  pages={503--528},
  year={1960},
  publisher={Cambridge University Press}
}

@article{kirby1978calculus,
  title={A calculus for framed links in S3},
  author={Kirby, Robion},
  journal={Invent. math},
  volume={45},
  number={1},
  pages={35--56},
  year={1978}
}

@article{kronheimer2004witten,
  title={Witten’s conjecture and property P},
  author={Kronheimer, Peter B and Mrowka, Tomasz S},
  journal={Geometry \& Topology},
  volume={8},
  number={1},
  pages={295--310},
  year={2004},
  publisher={Mathematical Sciences Publishers}
}

@article{gabai1987foliations,
  title={Foliations and the topology of 3-manifolds. III},
  author={Gabai, David},
  journal={Journal of Differential Geometry},
  volume={26},
  number={3},
  pages={479--536},
  year={1987},
  publisher={Lehigh University}
}
\sloppy 

\end{document}